\documentclass[11pt]{amsart}

\usepackage{tikz}

\usetikzlibrary{knots}
\usetikzlibrary{decorations.markings}

\usepackage{graphicx} 
\usepackage{t1enc}
\usepackage{amsthm,amssymb}
\usepackage{latexsym}
\usepackage{amsmath} 
\usepackage{graphics}
\usepackage{epsfig,multicol}
\usepackage{amsfonts}
\usepackage[latin1]{inputenc}
\usepackage[english]{babel}
\usepackage{ mathrsfs }

\newtheorem{theorem}{Theorem}[section]
\newtheorem*{theorem*}{Theorem}
\newtheorem*{proposition*}{Proposition}
\newtheorem{proposition}{Proposition}[section]
\newtheorem{lemma}{Lemma}[section]
\newtheorem*{lemma*}{Lemma}

\newtheorem{definition}{Definition}[section]

\newtheorem{corollary}{Corollary}[section]
\newtheorem{remark}{Remark}[section]

\def\hpic #1 #2 {\mbox{$\begin{array}[c]{l} \epsfig{file=#1,height=#2}
\end{array}$}}
 
\def\vpic #1 #2 {\mbox{$\begin{array}[c]{l} \epsfig{file=#1,width=#2}
\end{array}$}}

\newcommand  {\rmn}\romannumeral

\newcommand{\IC}[0]{\mathbb{C}} 

 \newcommand{\IN}[0]{\mathbb{N}}

 \newcommand{\IT}[0]{\mathbb{T}}

 \newcommand{\IZ}[0]{\mathbb{Z}}

\newcommand{\CA}[0]{\mathcal{A}} \newcommand{\CB}[0]{\mathcal{B}}
\newcommand{\CD}[0]{\mathcal{D}}
\newcommand{\CC}[0]{\mathcal{C}} \renewcommand{\CD}[0]{\mathcal{D}}
 \newcommand{\CF}[0]{\mathcal{F}}

\newcommand{\CO}[0]{\mathcal{O}} \newcommand{\CP}[0]{\mathcal{P}}
\newcommand{\CQ}[0]{\mathcal{Q}}

\begin{document}
\title[On the entropy and index of the winding endomorphisms of $\CQ_p$]{On the entropy and index of the winding endomorphisms of p-adic ring C$^*$-algebras}
\author{Valeriano Aiello} 
\address{Valeriano Aiello,
Mathematisches Institut, Universit\"at Bern, Alpeneggstrasse 22, 3012 Bern, Switzerland
}\email{valerianoaiello@gmail.com}
\author{Stefano Rossi} 
\address{Stefano Rossi,
Dipartimento di Matematica, Universit\`a degli studi Aldo Moro di Bari, Via E. Orabona 4, 70125 Bari, Italy
}\email{stefano.rossi@uniba.it}

\begin{abstract}

For $p\geq 2$, the $p$-adic ring $C^*$-algebra $\CQ_p$ is the universal $C^*$-algebra generated by a unitary $U$ and an isometry $S_p$ such that 
$S_pU=U^pS_p$ and $\sum_{l=0}^{p-1}U^lS_pS_p^*U^{-l}=1$. For any $k$ coprime with $p$  we define an endomorphism
$\chi_k\in{\rm End}(\CQ_p)$ by setting $\chi_k(U):=U^k$ and $\chi_k(S_p):=S_p$.
We then compute the entropy of $\chi_k$, which turns out to be $\log |k|$. Finally, for selected values of $k$ we also compute
the Watatani index of $\chi_k$ showing that the entropy is the natural logarithm of the index.

\end{abstract}

\maketitle

\section{Introduction}

First introduced by Adler, Konheim, and McAndrew \cite{AKM}, the topological entropy of a continuous map
on a compact Hausdorff space soon  proved to be a useful numerical invariant (under topological conjugacy)
to tackle, for instance, dynamics that may be out of the reach of the celebrated Halmos-von Neumann theorem, which only
settles those with topological discrete spectrum. Two results worth mentioning are that the entropy of any homeomorphism of the circle is null and the entropy of a differentiable map on a Riemannian manifold is finite, the latter of which is also known as
Kushnirenko's theorem. After a few years, Dinaburg and Bowen gave a novel yet equivalent definition for maps on metric spaces, which is particularly 
suited to establishing connections with Kolmogorov's measure theoretic entropy.\\
It was not until the mid $1990$s, though, that Voiculescu \cite{Voi} extended the original definition to
endomorphisms, or more generally to completely positive maps, of nuclear $C^*$-algebras, thought of
as the natural non-commutative counterpart of compact Hausdorff spaces. However, the computations involved
to find the exact value of the entropy are often rather demanding, so much so that not as many examples of endomorphisms 
as one would expect are known whose entropy has been computed. Of course, part of
the difficulty also depends on the choice of the $C^*$-algebra. Now the Cuntz algebras $\CO_p$, $p\geq 2$, are a natural family of $C^*$-algebras to consider not least because of their many connections with several research fields such as
algebraic quantum field theory, index theory, and wavelets. 
The first example to be discussed was the so-called canonical shift of $\CO_p$. In \cite{Choda} Choda showed that its entropy is
given by $\log p$, which is quite a remarkable fact as this value is nothing but the entropy of a Bernoulli shift on the alphabet
$\{1, 2, \ldots, p\}$ and the restriction of the canonical shift to the diagonal subalgebra $\CD_p\subset\CO_p$ is just such a Bernoulli shift. Soon after this result was obtained by Boca and Golstein \cite{Boca} for
 shift-type endomorphisms on arbitrary Cuntz-Krieger algebras by using a different technique, and more recently by 
Skalski and Zacharias \cite{SZ2} for  higher rank graph $C^*$-algebras. 
In \cite{SZ} the last-mentioned authors provided an upper bound to the entropy
of a general class of endomorphisms of $\CO_p$ that leave the UHF subalgebra $\CF_p$ invariant 
and satisfy a "finite-range" condition. Furthermore, they found the exact value of the entropy for all such endomorphisms
of $\CO_2$ associated  with permutations of rank $2$. In this paper we aim to show that a suitable
adaptation of the techniques employed in the aforementioned paper  can be exploited to compute the entropy
of a countable class of endomorphisms acting on the so-called $p$-adic ring $C^*$-algebras $\CQ_p$.
These and their generalizations have been of late the focus of much research \cite{ACR, ACR2, ACR3, ACR4, ACR5, ACRS, AR1, ACR6} and are here considered because they contain the Cuntz algebras in a natural way. Indeed, as we will see in the next section, each $\CO_p$ is contained in $\CQ_p$. Moreover, the commutative $C^*$-algebra of continuous functions
on the one-dimensional torus $\IT$ appears as a maximal abelian subalgebra of each $\CQ_p$. Now the endomorphisms
dealt with in our paper preserve this MASA, on which they simply act as $\IT\ni z\mapsto z^k\in\IT$, for some integer $k$.
For this reason we will refer to them as the \emph{winding endomorphisms}. Quite interestingly, their entropy is completely
determined by $k$. More precisely, the main result of the present paper is that their non-commutative entropy
is $\log |k|$, which is exactly the classical entropy of the continuous map $\Phi_k(z)=z^k$, $z\in\IT$, on the circle.
This is much in the same spirit as Choda's result on the entropy of the canonical shift we recalled above.\\
Finally, in the last section we  attack the problem of computing the Watatani index, too, of the winding endomorphisms 
so as to spot possible relations with the entropy, very much in line with what done in \cite{CS}, where the
quadratic permutation endomorphisms of the Cuntz algebra $\CO_2$ were studied.
The technique we employ can be applied only to values of $k$ of the form
$\pm(p-1)^i$, $i\in\IN$, and the Watatani index of the corresponding endomorphism turns out to be
exactly $|k|=(p-1)^i$. Nevertheless the index of the restriction of the winding endomorphisms to a remarkable
subalgebra of $\CQ_p$, the so-called gauge invariant subalgebra $\CQ_p^\IT$, which is isomorphic with
the Bunce-Deddens algebra of type $p^\infty$,
can be computed for all values of $k$ and, again, is given by $|k|$.
In particular, in all cases where the index can be computed the entropy is the natural logarithm of
the index.

\section{Preliminaries and notation}
Let $p$ be a natural number greater than or equal to $2$.
The $p$-adic ring C$^*$-algebra $\CQ_p$ is the universal $C^*$-algebra generated by a unitary $U$ and an isometry $S_p$ such that
$$
U^pS_p=S_pU \qquad \text{and} \qquad \sum_{l=0}^{p-1}U^lS_pS_p^*U^{-l}=1
$$
see also \cite{LarsenLi} for
$\CQ_2$,  and \cite{ACRS} for the general case. Note that 
$US_p^*=S_p^*U^p$, $U^*S_p^*=S_p^*U^{-p}$, and $\sum_{j=0}^{p^k}U^jS_p^k(S_p^*)^kU^{-j}=1$ for all $k\in\IN$. Furthermore, we also have $(S_p^*)^mU^{-i}U^jS_p^m=\delta_{i,j}$, if $0\leq i, j\leq p^m-1$.
All of these equalities can also be checked by means of the so-called canonical representation $\pi: \CQ_p\to \CB(\ell^2(\IZ))$ defined by 
$\pi(S_p)e_k:=e_{pk}$ and $\pi(U)e_k:=e_{k+1}$ for all $k\in\IZ$, where $\{e_k: k\in\IZ\}$ is the canonical basis of $\ell^2(\IZ)$, that is
$e_k(l):=\delta_{k,l}$ for any $k, l\in\IZ$. The canonical representation of a $p$-adic ring C$^*$-algebra
is irreducible, see \cite[Proposition 2.3]{ACRS}, where
the result is proved  for a broad class of
$C^*$-algebras, including all $p$-adic $C^*$-algebras, for which a canonical representation is always defined.
As is known, the Cuntz algebra $\CO_p$ is the universal
$C^*$-algebra generated by $p$ isometries $T_j$, $j=0, 1, \ldots, p-1$, such that $\sum_{j=0}^{p-1} T_j T_j^*=1$, \cite{Cuntz1}. We recall that $\CO_p$ injects into $\CQ_p$ through the $^*$-homomorphism that sends $T_j$ to $U^jS_p$ for $j=0,\ldots, p-1$. Henceforth, we will always think of $\CO_p$ as a subalgebra of $\CQ_p$.\\
The $p$-adic ring $C^*$-algebra is acted upon by $\IT$ in a natural way through the so-called
gauge automorphisms $\{\alpha_z: z\in\IT\}$. These are defined as $\alpha_z(S_p):=zS_p$ and
$\alpha_z(U):=U$. We denote by $\CQ_p^\IT\subset\CQ_p$ the subalgebra fixed by the gauge action of $\IT$, {\it i.e.}
$\CQ_p^\IT:=\{x\in\CQ_p: \alpha_z(x)=x,\,\,\textrm{for any}\,z\in\IT\}$. By definition, it is easy to check that $\CQ_p$ is
the norm closure of the linear span of monomials of the type $U^iS_p^h(S_p^*)^hU^j$, $h\in\IN$ and 
$i, j\in\IZ$.
Moreover, $\CQ_p^\IT$ is known to be isomorphic with the Bunce-Deddens algebra of type $p^\infty$, see
\cite[Remark 2.8]{BOSS}.
Another notable subalgebra of $\CQ_p$, which will play a key role in Section \ref{secindex}, is the so-called diagonal subalgebra,
$\CD_p$, which is the abelian $C^*$-algebra generated by all projections of the form $U^iS_p^m(S_p^*)^mU^{-i}$.
It turns out that $\CD_p$ is linearly generated by the above projections. Furthermore, $\CD_p$ is known to be maximal abelian
\cite{ACRS}. The Gelfand spectrum of $\CD_p$ can be seen to be homeomorphic with the Cantor set $K$, and the adjoint action
of $U$ restricts to $\CD_p$ as the $p$-adic odometer, which throughout this paper we denote by $T$.
The endomorphisms of $\CQ_p$ we will be focused on are those that fix $S_p$ while mapping $U$ to a power of it, say $U^k$. Set $\tilde U:=U^k$ and $\tilde S_p=S_p$. For such an endomorphism to exist, by universality it is necessary and
sufficient that $\tilde U$ and $\tilde S_p$ continue to satisfy the defining relations.
Now the relation $\tilde  U^p\tilde S_p=\tilde S_p\tilde U$ does not cause any restriction on $k$ since it is trivially satisfied.
 Because $U^p$ commutes with $S_pS_p^*$, the relation $\sum_{l=0}^{p-1}\tilde U^l\tilde S_p\tilde S_p^*\tilde U^{-l}=1$ does entail a restriction on the possible values of $k$, for we must have $\{[0],[k],[2k], \ldots, [(p-1)k]\}=\IZ_p$, where
$[l]$ denotes the congruence class of $l$ modulo $p$. 
This condition is fulfilled if and only if $k$ and $p$ are coprime, namely when their
greatest common divisor is $1$, in which case we write $(k, p)=1$. This is a consequence of a simple result, which we single out below
for the reader's convenience.
\begin{proposition}
Let $p>1$ be a fixed integer number. Then the group homomorphism $\Psi_k$ defined on
$(\IZ_p, +)$ by $\Psi_k ([n]):=[kn]$, for any $[n]\in\IZ_p$, is surjective if and only if $(k, p)=1$
\end{proposition}

Thus, for any $k$ coprime with $p$ we can introduce the \emph{winding endomorphisms} $\chi_{k}:\CQ_p\to\CQ_p$ given by $\chi_{k}(U):=U^{k}$, $\chi_{k}(S_p):=S_p$. Except when $k= \pm 1$, these are all proper endomorphisms, cf. \cite[Proposition 6.1]{ACR} . When $p=2$, these endomorphisms were originally introduced in \cite[Section 6]{ACR} for $\CQ_2$. Note that $\chi_{k_1}\circ\chi_{k_2}=\chi_{k_1k_2}$, for any pair of integers
$k_1, k_2$ coprime with $p$.

\medskip
We now recall Voiculescu's definition of topological entropy, \cite[Section 4]{Voi}.
Since the $C^*$-algebras dealt with in this paper are all unital and nuclear, we will limit ourselves to  recalling the definition
for this class, although a more general definition can be given for arbitrary exact $C^*$-algebras, see \cite{B}.\\
Given a {\it nuclear} $C^*$-algebra  $\CA$ and an endomorphism  $\alpha: \CA\to \CA$, we denote by 
 CPA$(\CA)$ the set of triples $(\phi, \psi, \CB)$, where $\CB$ is a finite-dimensional $C^*$-algebra, $\phi: \CA\to \CB$, $\psi: \CB\to\CA$ are unital {\it completely positive} maps (u.c.p. for short).
For any $\epsilon>0$ and any finite subset $\omega\subset \CA$ (for brevity we write $\omega\in \CP f(\CA)$), 
we denote by CPA$(\CA,\omega, \epsilon)$ the set of triples $(\phi, \psi, \CB)\in {\rm CPA}(\CA)$ such that  $\| (\psi\circ\phi)(a)-a\|<\epsilon$  for all  $a\in\omega$.
As is known, the nuclearity of $\CA$ is equivalent  to the existence of a triple $(\phi,\psi,\CB)\in {\rm CPA(\CA, \omega,\epsilon)}$ for any $\omega\in\CP f(\CA)$ and $\epsilon>0$. For a thorough account of completely positive maps and
nuclear (also known as amenable) $C^*$-algebras, we refer the reader to \cite{HuaxinLin}.\\
The completely positive $\epsilon$-rank of an endomorphism $\alpha$ is then  defined by the following formula
$$
{\rm rcp}(\omega,\epsilon):=\inf\{{\rm rank}(\CB) \; | \; (\phi, \psi, \CB)\in {\rm CPA}(\CA,\omega, \epsilon)\}
$$
where rank$(\CB)$ denotes the dimension of a maximal abelian subalgebra of $\CB$.
If we set
\begin{align*}
{\rm ht}(\alpha, \omega;\epsilon) & :=\limsup_{n\to\infty} \frac{\log{\rm rcp}(\omega\cup\alpha(\omega)\cup\ldots\cup\alpha^{n-1}(\omega);\epsilon)}{n}\\
{\rm ht}(\alpha, \epsilon) & :=\sup_{\epsilon>0}{\rm ht}(\alpha, \omega;\epsilon)
\end{align*}
 the topological entropy of $\alpha$ is finally defined as
$$
{\rm ht}(\alpha):=\sup_{\omega\in\CP f(\CA)}{\rm ht}(\alpha, \omega)
$$
One way to obtain a lower bound for the topological entropy is to consider a commutative $C^*$-algebra $\CC$ of $\CA$ that is invariant under $\alpha$. 
Then, it holds 
$$
{\rm ht}(\alpha)\geq {\rm ht}(\alpha\upharpoonright_\CC)={\rm h}_{\rm top}(T)
$$ 
where
$T$ is the map induced by  $\alpha\upharpoonright_\CC$ on the level of the spectrum of $\CC$,  \cite{Voi}.
Sometimes a lower bound thus obtained is just the exact value of the entropy. 
However, in \cite{Adam} examples are given of automorphisms on non-commutative
$C^*$-algebras whose entropy is in fact bigger than the supremum of the the lower bounds provided by considering the restriction to
all classical subsystems.
Another fundamental tool is the so-called Kolmogorov-Sinai property which says that if $(\omega_i)_{i\in I}$ is a family of finite subsets of $\CA$ such that the linear span of $\cup_{i\in I, n\in\IN}\alpha^n(\omega_i)$ 
is dense in $\CA$, then 
$$
{\rm ht}(\alpha)=\sup_{\epsilon>0, i\in I}\limsup_{n\to\infty}\left( \frac{1}{n}\log{\rm rcp}(\alpha^n(\omega_i),\epsilon)	\right)
$$


\section{Main result}

 \begin{theorem}\label{entropy}
For any $k$ coprime with $p$, the entropy of the winding endomorphims is given by
$$
{\rm ht}(\chi_{k})= \log|k|\; .
$$
\end{theorem}

The proof requires some technical preliminary results, which are given below.
First,   introduce a countable family of finite sets whose linear span coincides with the whole
$p$-adic ring $C^*$-algebra.\\
For any $l,m,n \in\IN$,  the set $\CA_{l,m,n}$ is by definition the set of all monomials of the form $U^iS_p^m(S_p^*)^n U^j$, where $|i|$, $|j|\leq l$,  and  one of the following three conditions holds
\begin{enumerate}
\item $p^m> p^n> |j|$. 
\item $|i|< p^m< p^n$.
\item $p^m= p^n> |j|$.
\end{enumerate}

Finally, $\CB_{l,m,n}$ is the the vector space generated by $\CA_{l,m,n}$.

\begin{remark}\label{exponents}
The sets $\CA_{l,m,n}$ are mapped to $\CA_{lk,m,n}$ by the  winding endomorphism $\chi_{k}\in{\rm Aut}(\CQ_p)$.
Indeed, we have $\chi_{k}(U^iS_p^m(S_p^*)^nU^j)=U^{ik}S_p^m(S_p^*)^nU^{kj}$.
This means that the vector spaces $\CB_{l,m,n}$, too, are mapped to $\CB_{kl,m,n}$  by $\chi_k$.
\end{remark}

\begin{lemma}
The set $\cup_{l,m,n=0}^\infty \CA_{l,m,n }$  linearly generates a dense subspace of $\CQ_p$.
\end{lemma}
\begin{proof}
The monomials $\{U^iS_p^m(S_p^*)^n U^j\; | \; i, j \in\IZ, m,n\in\IN\}$ generate a dense subspace of $\CQ_p$ (see \cite[Section 2]{ACRS} and the references therein). The fact that we only need to consider the three aforementioned cases  is explained below. \\
If $p^m\geq p^n\leq |j|$, then $j=p^na+b$ (with $|b|<p^n$) and  
$$
U^iS_p^m(S_p^*)^n U^j=U^iS_p^m(S_p^*)^n U^{p^na+b}=U^iS_p^mU^a(S_p^*)^n U^{b}=U^{i+p^ma}S_p^m(S_p^*)^n U^{b}
$$
 If $|i|\geq p^m< p^n$, then $i=p^ma+b$ (with $|b|<p^m$) and  
$$
U^iS_p^m(S_p^*)^n U^j=U^{p^ma+b}S_p^m(S_p^*)^n U^j=U^{b}S_p^mU^a(S_p^*)^n U^j =U^{b}S_p^m(S_p^*)^n U^{j+p^na} 
$$
where we used $US_p^*=S_p^*U^p$.\\
If $p^m= p^n\leq |j|$, then $j=p^ma+b$ (with $|b|<p^m$) and 
$$
U^iS_p^m(S_p^*)^m U^j=U^iS_p^m(S_p^*)^m U^{p^ma+b}=U^iU^{p^ma}S_p^m(S_p^*)^m U^{b}
$$
where we used $U^{p^m}S_p^m(S_p^*)^m=S_p^m(S_p^*)^mU^{p^m}$.
\end{proof}

In the sequel we will repeatedly make use of the natural identification between $M_n(\IC)\otimes \CQ_p$  and $M_n(\CQ_p)$.\\
In the following lemma we single out an isomorphism between the $p$-adic ring
$C^*$-algebra $\CQ_p$ and its tensor product with  $p^h\times p^h$ matrices. This 
will be useful  in some of the subsequent computations.

\begin{lemma}
For any $p\geq 2$ and for any $h\geq 1$, the map $\Psi_h: \CQ_p\to M_{p^h}(\IC)\otimes \CQ_p$ given by $\Psi_h(x):=\sum_{i,j=0}^{p^h-1} e_{i,j}\otimes (S_p^*)^{h}U^{-i} x U^jS_p^h$, $x\in\CQ_p$, is an isomorphism.
\end{lemma}
\begin{proof}
It is enough to check that the map is multiplicative
\begin{align*}
\Psi_h(x)\Psi_h(y)&=\left( \sum_{i,j=0}^{p^h-1} e_{i,j}\otimes   (S_p^*)^{h}U^{-i} x U^jS_p^h\right)\left( \sum_{m,n=0}^{p^h-1} e_{m,n}  \otimes (S_p^*)^{h}U^{-m} y U^nS_p^h\right)\\
&= \sum_{i,j,m,n=0}^{p^h-1} e_{i,j}e_{m,n} \otimes   (S_p^*)^{h}U^{-i} x U^jS_p^h  (S_p^*)^{h}U^{-m} y U^nS_p^h\\
&= \sum_{i,j,m,n=0}^{p^h-1} \delta_{j,m} e_{i,j}e_{m,n} \otimes   (S_p^*)^{h}U^{-i} x U^j S_p^h (S_p^*)^{h}U^{-m} y U^nS_p^h\\
&= \sum_{i,j,n=0}^{p^h-1} e_{i,n} \otimes   (S_p^*)^{h}U^{-i} x U^jS_p^h  (S_p^*)^{h}U^{-j} y U^n S_p^h\\
&= \sum_{i,n=0}^{p^h-1} e_{i,n} \otimes   (S_p^*)^{h}U^{-i} x \left(\sum_{j=0}^{p^h-1} U^jS_p^h  (S_p^*)^{h}U^{-j} \right) y U^n S_p^h\\
&= \sum_{i,n=0}^{p^h-1} e_{i,n} \otimes   (S_p^*)^{h}U^{-i} xy U^n S_p^h\\
&=\Psi_h(xy)
\end{align*}
Injectivity follows from the simplicity of $\CQ_p$.\\
As for the surjectivity, it suffices to show that, for all $x\in\CQ_p$ and $i,j$, the element $e_{i,j}\otimes x$ is in the image of $\Psi_h$.
Indeed, we have
\begin{align*}
\Psi_h(U^iS_p^hx(S_p^h)^*U^{-j})&=\sum_{i',j'=0}^{p^h-1} e_{i',j'}\otimes (S_p^*)^{h}U^{-i'} (U^iS_p^hx(S_p^h)^*U^{-j}) U^{j'}S_p^h\\
&=e_{i,j}\otimes x
\end{align*} 
where in last step we used that $\{U^iS_p^h: i=0, 1, \ldots, p^h-1\}$ is a family of mutually orthogonal isometries,
as can be checked in the canonical representation. Indeed, for any $k\in\IZ$ we have
$$(S_p^*)^hU^{-j}U^iS_p^he_k=(S_p^*)^h e_{kp^h+(i-j)}$$
which means $(S_p^*)^hU^{-j}U^iS_p^he_k=0$ if $i\neq j$ because $i-j$ is never a multiple of
$p^h$.
\end{proof}
In the following lemma we point out an inequality which will come in useful in the next. 
\begin{lemma}\label{normineq}
Let $d$ be an integer number and $\{Q_i\}_{i=1}^5$, $\{R_j\}_{j=1}^3\subset M_d(\IC)$. If we define $A, B \in M_d(\CQ_p)$ as
\begin{eqnarray*}
&& A:=Q_1\otimes S_p^{m-n}+Q_2\otimes S_p^{m-n}U+ Q_3\otimes S_p^{m-n}U^*+ Q_4\otimes US_p^{m-n}+ Q_5\otimes U^*S_p^{m-n}\\
&& B:= R_1\otimes 1+  R_2\otimes U+R_3\otimes U^*
\end{eqnarray*}
for any pair of integer numbers with $m>n$,
we have $\|Q_i\|\leq \|A\|$ for any $i=1, 2, 3, 4, 5$ and $\|R_j\|\leq \|B\|$ for any $j=1, 2, 3$.
\end{lemma}
\begin{proof}
We will only treat $A$, for $B$ can be dealt with even more easily.
We think of $M_d(\CQ_p)\cong M_d(\IC)\otimes \CQ_p$ as a concrete
$C^*$-algebra acting on the tensor Hilbert space $\IC^d\otimes\ell^2(\IZ)$.
For any $x, x'\in \IC^d$ and $y, y'\in \ell^2(\IZ)$  with 
$\|x\|, \|x'\|, \|y\|, \|y'\|\leq 1$
we have:
\begin{eqnarray*}
&&\|A\|\geq |(A(x\otimes y), x'\otimes y')|=\\
&& | (Q_1x, x') (S_p^{m-n}y, y')+   (Q_2x, x') (S_p^{m-n}Uy, y')+(Q_3 x, x') (S_p^{m-n}U^*y, y')\\
&&\qquad +(Q_4x, x') (US_p^{m-n}y, y')+(Q_5x, x') (U^*S_p^{m-n}y, y')  |
\end{eqnarray*} 
There are now five cases to consider.
\begin{itemize}
\item Choosing $y=e_1$ and $y'=S_p^{m-n}e_1= e_{p^{m-n}}$ the inequality simply becomes
$\|A\|\geq |(Q_1x, x')|$ as the remaining four terms are separately zero as the product of two factors, the second of which vanishes by
construction. Taking the sup on $x, x'$ running
on the unit ball of $\IC^d$ the inequality in the statement is obtained.
\item Choosing $y=e_1$ and $y'=S_p^{m-n}Ue_1=e_{2 p^{(m-n)}}$ we now find $\|A\|\geq |(Q_2x, x')|$ and the conclusion follows.
\item Choosing $y=e_1$ and $y'=S_p^{m-n}U^*e_1=e_0$ we now find $\|A\|\geq |(Q_3x, x')|$ and the conclusion follows.
\item Choosing $y=e_1$ and $y'=US_p^{m-n}e_1=e_{p^{m-n}+1}$ we now find $\|A\|\geq |(Q_4x, x')|$ and the conclusion follows.
\item Choosing $y=e_1$ and $y'=U^*S_p^{m-n}e_1=e_{ p^{(m-n)}-1}$ we now find $\|A\|\geq |(Q_5x, x')|$ and the conclusion follows.
\end{itemize}
\end{proof}
The following lemma is one of the main ingredients in the proof of an upper bound for the entropy of the winding endomorphisms.
\begin{lemma}\label{lemma21}
Let $h, l,m,n\in\IN$, with $h>\max\{m,n\}$, $l< p^{h}$, and $x\in \CB_{l,m,n}$. We have:
\begin{itemize}
\item If $m>n$, then $\Psi_h(x)= \sum_{j=1}^{p^{m-n}-1} R_{j}\otimes U^{j}S_p^{m-n}+R_{0}\otimes S_p^{m-n}+\tilde R_{0}\otimes S_p^{m-n}U+\hat R_{0}\otimes U^*S_p^{m-n}$ where 
$R_{j}, R_0, \tilde{R}_0, \hat{R}_0 \in M_{p^h}(\IC)$, with $\| R_{j}\|\leq \| x\|$, $\| R_{0}\|\leq \| x\|$, $\| \tilde R_{0}\|\leq \| x\|$, $\| \hat R_{0}\|\leq \| x\|$.
\item If $m=n$, then $\Psi_h(x)=R_1\otimes 1+ R_2\otimes U+R_3\otimes U^*$ where 
$R_1,  R_2, R_3\in M_{p^h}(\IC)$, and $\| R_i\|\leq  \| x\|$ for all $i$.
\item If $m<n$, then $\Psi_h(x)= \sum_{j=1}^{p^{n-m}-1} R_{j}\otimes (S_p^{n-m})^*U^{-j}+R_{0}\otimes S_p^{n-m}+\tilde R_{0}\otimes S_p^{m-n}U+\hat R_{0}\otimes U^*S_p^{m-n}$ where 
$R_{j}, R_0, \tilde{R}_0, \hat{R}_0\in M_{p^h}(\IC)$, $\| R_{j}\|\leq \| x\|$, with $\| R_{0}\|\leq \| x\|$, $\| \tilde R_{0}\|\leq \| x\|$, $\| \hat R_{0}\|\leq \| x\|$.
\end{itemize}
\end{lemma}
\begin{proof}
Without loss of generality, we may suppose that $x=U^aS_p^m(S_p^*)^nU^d$. 
We start by dealing with the first case, that is $m>n$. In turn we will have to settle
four subcases depending on the signs of $a$ and $d$.\\

Suppose  that $a\geq 0$, $d\leq 0$, $|d|<p^n$. 
We have $a=p^mb+r$ (with $0\leq r<p^m$).
Note that $0\leq b<p^{h-m}$. 
Then we have
\begin{align*}
\Psi_h(x)&=\sum_{i,j=0}^{p^h-1} e_{i,j}\otimes (S_p^*)^{h}U^{-i} U^aS_p^m(S_p^*)^nU^d U^jS_p^h\\
&=\sum_{i,j=0}^{p^h-1} e_{i,j}\otimes (S_p^*)^{h-m} ((S_p^*)^{m}U^{-i} U^{r}U^{p^mb}S_p^m)(S_p^*)^nU^d U^jS_p^h\\
&=\sum_{i,j=0}^{p^h-1} e_{i,j}\otimes (S_p^*)^{h-m} ((S_p^*)^{m}U^{-i} U^rS_p^m)U^b(S_p^*)^nU^d U^jS_p^h\\
&=\sum_{j=0}^{p^h-1}\sum_{i_1=0}^{p^m-1}\sum_{i_2=0}^{p^{h-m}-1} e_{i_1+i_2p^m,j}\otimes (S_p^*)^{h-m} ((S_p^*)^{m}U^{-(i_1+i_2p^m)} U^rS_p^m)U^b(S_p^*)^nU^d U^jS_p^h\\
&=\sum_{j=0}^{p^h-1}\sum_{i_1=0}^{p^m-1}\sum_{i_2=0}^{p^{h-m}-1} e_{i_1+i_2p^m,j}\otimes (S_p^*)^{h-m} U^{-i_2} ((S_p^*)^{m}U^{-i_1} U^rS_p^m)U^b(S_p^*)^nU^d U^jS_p^h\\
&=\sum_{j=0}^{p^h-1}\sum_{i_2=0}^{p^{h-m}-1}e_{r+i_2p^m,j}\otimes (S_p^*)^{h-m}   U^{-i_2+b} (S_p^*)^nU^d U^jS_p^h \\
&=\sum_{j=0}^{p^h-1} \sum_{i_2=0}^{p^{h-m}-1}e_{r+i_2p^m,j}\otimes (S_p^*)^{h-m} U^{-i_2+b} ((S_p^*)^nU^d U^jS_p^n) S_p^{h-n}\\
&=\sum_{j_1=0}^{p^n-1}\sum_{j_2=0}^{p^{h-n}-1} \sum_{i_2=0}^{p^{h-m}-1}e_{r+i_2p^m,j_1+p^{n}j_2}\otimes (S_p^*)^{h-m} U^{-i_2+b} ((S_p^*)^nU^d U^{j_1+p^{n}j_2}S_p^n) S_p^{h-n}\\
&=\sum_{j_1=0}^{p^n-1}\sum_{j_2=0}^{p^{h-n}-1} \sum_{i_2=0}^{p^{h-m}-1}e_{r+i_2p^m,j_1+p^{n}j_2}\otimes (S_p^*)^{h-m} U^{-i_2+b} ((S_p^*)^nU^d U^{j_1}S_p^n)U^{j_2} S_p^{h-n}\\
&=\sum_{j_2=0}^{p^{h-n}-1} \sum_{i_2=0}^{p^{h-m}-1}e_{r+i_2p^m,-d+p^{n}j_2}\otimes (S_p^*)^{h-m} U^{-i_2+b} U^{j_2} S_p^{h-n}\\
&=\sum_{j_2=0}^{p^{h-n}-1} \sum_{i_2=0}^{p^{h-m}-1}e_{r+i_2p^m,-d+p^{n}j_2}\otimes (S_p^*)^{h-m} U^{-i_2+b+j_2} S_p^{h-m}S_p^{m-n}\\
&=\sum_{j_3=0}^{p^{h-m}-1}\sum_{j_4=0}^{p^{m-n}-1} \sum_{i_2=0}^{p^{h-m}-1}e_{r+i_2p^m,-d+p^{n}(j_3+j_4p^{h-m})}\otimes (S_p^*)^{h-m} U^{-i_2+b} U^{j_3+j_4p^{h-m}} S_p^{h-m}S_p^{m-n}\\
&= \sum_{j_3=0}^{p^{h-m}-1}\sum_{j_4=0}^{p^{m-n}-1} \sum_{i_2=0}^{p^{h-m}-1}e_{r+i_2p^m,-d+p^{n}(j_3+j_4p^{h-m})}\otimes (S_p^*)^{h-m} U^{-i_2+b} U^{j_3} S_p^{h-m}U^{j_4}S_p^{m-n}\\
\end{align*}
Now since $-p^{h-m}+1\leq-i_2+b+j_3\leq 2p^{h-m}-2$, there is only one nontrivial multiple of
$p^{h-m}$ among the values taken by  $-i_2+b+j_3$, namely $p^{h-m}$ itself. Therefore, the last expression we had can be rewritten as
\begin{align*}
& \sum_{j_4=0}^{p^{m-n}-1} \sum_{i_2=0}^{p^{h-m}-1}e_{r+i_2p^m,-d+p^{n}(i_2-b+j_4p^{h-m})}\otimes U^{j_4}S_p^{m-n} +\\
&\qquad + \sum_{j_4=0}^{p^{m-n}-1} \sum_{i_2=0}^{p^{h-m}-1}e_{r+i_2p^m,-d+p^{n}(i_2-b+p^{h-m}+j_4p^{h-m})}\otimes U^{j_4+1}S_p^{m-n} \\
&= \sum_{j_4=0}^{p^{m-n}-1} \sum_{i_2=0}^{p^{h-m}-1}e_{r+i_2p^m,-d+p^{n}(i_2-b+j_4p^{h-m})}\otimes U^{j_4}S_p^{m-n} \\
&\qquad + \sum_{j_4=1}^{p^{m-n}} \sum_{i_2=0}^{p^{h-m}-1}e_{r+i_2p^m,-d+p^{n}(i_2-b+p^{h-m}+j_4p^{h-m}-p^{h-m})}\otimes U^{j_4}S_p^{m-n} \\
&= \sum_{j_4=0}^{p^{m-n}-1} \sum_{i_2=0}^{p^{h-m}-1}e_{r+i_2p^m,-d+p^{n}(i_2-b+j_4p^{h-m})}\otimes U^{j_4}S_p^{m-n} \\
&\qquad + \sum_{j_4=1}^{p^{m-n}-1} \sum_{i_2=0}^{p^{h-m}-1}e_{r+i_2p^m,-d+p^{n}(i_2-b+p^{h-m}+j_4p^{h-m}-p^{h-m})}\otimes U^{j_4}S_p^{m-n} \\
&\qquad + \sum_{i_2=0}^{p^{h-m}-1}e_{r+i_2p^m,-d+p^{n}(i_2-b+p^{h-m}+p^{m-n}p^{h-m}-p^{h-m})}\otimes U^{p^{m-n}}S_p^{m-n} \\
\end{align*}
and in the last sum of the above expression we easily recognize a term of the form $\tilde R_{0}\otimes S_p^{m-n}U$.


Suppose that $a\leq 0$, $d\leq 0$, $|d|<p^n$. 
We have $a=p^mb+r$ (with $0\leq r<p^m$).
Note that $0\leq  -b\leq p^{h-m}$. 
Then we have
\begin{align*}
\Psi_h(x)&=\sum_{i,j=0}^{p^h-1} e_{i,j}\otimes (S_p^*)^{h}U^{-i} U^aS_p^m(S_p^*)^nU^d U^jS_p^h\\
&=\sum_{i,j=0}^{p^h-1} e_{i,j}\otimes (S_p^*)^{h}U^{-i} U^rS_p^mU^b(S_p^*)^nU^d U^jS_p^h\\
\end{align*}
\begin{align*}
&=\sum_{i,j=0}^{p^h-1} e_{i,j}\otimes (S_p^*)^{h-m} ((S_p^*)^{m}U^{-i} U^rS_p^m)U^b(S_p^*)^nU^d U^jS_p^h\\
&=\sum_{j=0}^{p^h-1}\sum_{i_1=0}^{p^m-1}\sum_{i_2=0}^{p^{h-m}-1} e_{i_1+i_2p^m,j}\otimes (S_p^*)^{h-m} ((S_p^*)^{m}U^{-(i_1+i_2p^m)} U^rS_p^m)U^b(S_p^*)^nU^d U^jS_p^h\\
&=\sum_{j=0}^{p^h-1}\sum_{i_1=0}^{p^m-1}\sum_{i_2=0}^{p^{h-m}-1} e_{i_1+i_2p^m,j}\otimes (S_p^*)^{h-m} U^{-i_2} ((S_p^*)^{m}U^{-i_1} U^rS_p^m)U^b(S_p^*)^nU^d U^jS_p^h\\
&=\sum_{j=0}^{p^h-1}\sum_{i_2=0}^{p^{h-m}-1}e_{r+i_2p^m,j}\otimes (S_p^*)^{h-m}   U^{-i_2+b} (S_p^*)^nU^d U^jS_p^h \\
&=\sum_{j=0}^{p^h-1} \sum_{i_2=0}^{p^{h-m}-1}e_{r+i_2p^m,j}\otimes (S_p^*)^{h-m} U^{-i_2+b} ((S_p^*)^nU^d U^jS_p^n) S_p^{h-n}\\
&=\sum_{j_1=0}^{p^n-1}\sum_{j_2=0}^{p^{h-n}-1} \sum_{i_2=0}^{p^{h-m}-1}e_{r+i_2p^m,j_1+p^{n}j_2}\otimes (S_p^*)^{h-m} U^{-i_2+b} ((S_p^*)^nU^d U^{j_1+p^{n}j_2}S_p^n) S_p^{h-n}\\
&=\sum_{j_1=0}^{p^n-1}\sum_{j_2=0}^{p^{h-n}-1} \sum_{i_2=0}^{p^{h-m}-1}e_{r+i_2p^m,j_1+p^{n}j_2}\otimes (S_p^*)^{h-m} U^{-i_2+b} ((S_p^*)^nU^d U^{j_1}S_p^n)U^{j_2} S_p^{h-n}\\
&=\sum_{j_2=0}^{p^{h-n}-1} \sum_{i_2=0}^{p^{h-m}-1}e_{r+i_2p^m,-d+p^{n}j_2}\otimes (S_p^*)^{h-m} U^{-i_2+b+j_2} S_p^{h-m}S_p^{m-n}\\
&=\sum_{j_3=0}^{p^{h-m}-1}\sum_{j_4=0}^{p^{m-n}-1} \sum_{i_2=0}^{p^{h-m}-1}e_{r+i_2p^m,-d+p^{n}(j_3+j_4p^{h-m})}\otimes (S_p^*)^{h-m} U^{-i_2+b} U^{j_3+j_4p^{h-m}} S_p^{h-m}S_p^{m-n}\\
&= \sum_{j_3=0}^{p^{h-m}-1}\sum_{j_4=0}^{p^{m-n}-1} \sum_{i_2=0}^{p^{h-m}-1}e_{r+i_2p^m,-d+p^{n}(j_3+j_4p^{h-m})}\otimes (S_p^*)^{h-m} U^{-i_2+b} U^{j_3} S_p^{h-m}U^{j_4}S_p^{m-n}
\end{align*}
Now since $-2p^{h-m}+1\leq b+j_3-i_2\leq p^{h-m}-1$, there is only one nontrivial multiple of
$p^{h-m}$ among the values taken by  $-i_2+b+j_3$, namely -$p^{h-m}$ itself. Therefore, the last expression we had can be rewritten as
\begin{align*}
&\sum_{j_4=0}^{p^{m-n}-1} \sum_{i_2=0}^{p^{h-m}-1}e_{r+i_2p^m,-d+p^{n}(i_2-b+j_4p^{h-m})}\otimes U^{j_4}S_p^{m-n} \\
&\qquad + \sum_{j_4=0}^{p^{m-n}-1} \sum_{i_2=0}^{p^{h-m}-1}e_{r+i_2p^m,-d+p^{n}(i_2-b-p^{h-m}+j_4p^{h-m})}\otimes U^{j_4-1}S_p^{m-n} \\
\end{align*}
Now in the second sum of the above expression the summand corresponding to $j_4=0$ accounts for the presence of a term
of the type $\hat R_{0}\otimes U^*S_p^{m-n}$, as in the statement.\\

Now we assume $a\geq 0$, $d\geq 0$, $|d|<p^n$. 
We have $a=p^mb+r$ (with $0\leq r<p^m$).
Note that $0\leq  b<p^{h-m}$. 
Then we have
\begin{align*}
&\Psi_h(x)=\sum_{i,j=0}^{p^h-1} e_{i,j}\otimes (S_p^*)^{h}U^{-i} U^aS_p^m(S_p^*)^nU^d U^jS_p^h\\
&=\sum_{i,j=0}^{p^h-1} e_{i,j}\otimes (S_p^*)^{h-m} ((S_p^*)^{m}U^{-i} U^rS_p^m)U^b(S_p^*)^nU^d U^jS_p^h\\
&=\sum_{j=0}^{p^h-1}\sum_{i_1=0}^{p^m-1}\sum_{i_2=0}^{p^{h-m}-1} e_{i_1+i_2p^m,j}\otimes (S_p^*)^{h-m} ((S_p^*)^{m}U^{-(i_1+i_2p^m)} U^rS_p^m)U^b(S_p^*)^nU^d U^jS_p^h\\
&=\sum_{j=0}^{p^h-1}\sum_{i_1=0}^{p^m-1}\sum_{i_2=0}^{p^{h-m}-1} e_{i_1+i_2p^m,j}\otimes (S_p^*)^{h-m} U^{-i_2} ((S_p^*)^{m}U^{-i_1} U^rS_p^m)U^b(S_p^*)^nU^d U^jS_p^h\\
&=\sum_{j=0}^{p^h-1}\sum_{i_2=0}^{p^{h-m}-1}e_{r+i_2p^m,j}\otimes (S_p^*)^{h-m}   U^{-i_2+b} (S_p^*)^nU^d U^jS_p^h \\
&=\sum_{j=0}^{p^h-1} \sum_{i_2=0}^{p^{h-m}-1}e_{r+i_2p^m,j}\otimes (S_p^*)^{h-m} U^{-i_2+b} ((S_p^*)^nU^d U^jS_p^n) S_p^{h-n}\\
&=\sum_{j_1=0}^{p^n-1}\sum_{j_2=0}^{p^{h-n}-1} \sum_{i_2=0}^{p^{h-m}-1}e_{r+i_2p^m,j_1+p^{n}j_2}\otimes (S_p^*)^{h-m} U^{-i_2+b} ((S_p^*)^nU^d U^{j_1+p^{n}j_2}S_p^n) S_p^{h-n}\\
&=\sum_{j_1=0}^{p^n-1}\sum_{j_2=0}^{p^{h-n}-1} \sum_{i_2=0}^{p^{h-m}-1}e_{r+i_2p^m,j_1+p^{n}j_2}\otimes (S_p^*)^{h-m} U^{-i_2+b} ((S_p^*)^nU^d U^{j_1}S_p^n)U^{j_2} S_p^{h-n}\\
&=\sum_{j_2=0}^{p^{h-n}-1} \sum_{i_2=0}^{p^{h-m}-1}e_{r+i_2p^m,p^n-d+p^{n}j_2}\otimes (S_p^*)^{h-m} U^{-i_2+b} U^{j_2+1} S_p^{h-n}\\
&=\sum_{j_2=0}^{p^{h-n}-1} \sum_{i_2=0}^{p^{h-m}-1}e_{r+i_2p^m,p^n-d+p^{n}j_2}\otimes (S_p^*)^{h-m} U^{-i_2+b+j_2+1} S_p^{h-m}S_p^{m-n}\\
&=\sum_{j_3=0}^{p^{h-m}-1}\sum_{j_4=0}^{p^{m-n}-1} \sum_{i_2=0}^{p^{h-m}-1}e_{r+i_2p^m,p^n-d+p^{n}(j_3+j_4p^{h-m})}\otimes (S_p^*)^{h-m} U^{-i_2+b+1} U^{j_3+j_4p^{h-m}} S_p^{h-m}S_p^{m-n}\\
\end{align*}
\begin{align*}
&= \sum_{j_3=0}^{p^{h-m}-1}\sum_{j_4=0}^{p^{m-n}-1} \sum_{i_2=0}^{p^{h-m}-1}e_{r+i_2p^m,p^n-d+p^{n}(j_3+j_4p^{h-m})}\otimes (S_p^*)^{h-m} U^{-i_2+b+1} U^{j_3} S_p^{h-m}U^{j_4}S_p^{m-n}
\end{align*}
Now since $-p^{h-m}+2\leq  -i_2+b+1+j_3 \leq 2p^{h-m}-1$, there is only one nontrivial multiple of
$p^{h-m}$ among the values taken by  $-i_2+b+j_3$, namely $p^{h-m}$ itself. Therefore, the last expression we had can be rewritten as
\begin{align*}
& \sum_{j_4=0}^{p^{m-n}-1} \sum_{i_2=0}^{p^{h-m}-1}e_{r+i_2p^m,p^n-d+p^{n}(i_2-b-1+j_4p^{h-m})}\otimes U^{j_4}S_p^{m-n} \\
&\qquad + \sum_{j_4=0}^{p^{m-n}-1} \sum_{i_2=0}^{p^{h-m}-1}e_{r+i_2p^m,p^n-d+p^{n}(i_2-b-1+p^{h-m}+j_4p^{h-m})}\otimes U^{j_4+1}S_p^{m-n} 
\end{align*}

Finally, we discuss the forth subcase: $a\leq 0$, $d\geq 0$, $d<p^n$. 
We have $a=p^mb+r$ (with $0\leq r<p^m$).
Note that $0\leq  -b\leq p^{h-m}$. 
We have
\begin{align*}
\Psi_h(x)&=\sum_{i,j=0}^{p^h-1} e_{i,j}\otimes (S_p^*)^{h}U^{-i} U^aS_p^m(S_p^*)^nU^d U^jS_p^h\\
&=\sum_{i,j=0}^{p^h-1} e_{i,j}\otimes (S_p^*)^{h-m} ((S_p^*)^{m}U^{-i} U^rS_p^m)U^b(S_p^*)^nU^d U^jS_p^h\\
&=\sum_{j=0}^{p^h-1}\sum_{i_1=0}^{p^m-1}\sum_{i_2=0}^{p^{h-m}-1} e_{i_1+i_2p^m,j}\otimes (S_p^*)^{h-m} ((S_p^*)^{m}U^{-(i_1+i_2p^m)} U^rS_p^m)U^b(S_p^*)^nU^d U^jS_p^h\\
&=\sum_{j=0}^{p^h-1}\sum_{i_1=0}^{p^m-1}\sum_{i_2=0}^{p^{h-m}-1} e_{i_1+i_2p^m,j}\otimes (S_p^*)^{h-m} U^{-i_2} ((S_p^*)^{m}U^{-i_1} U^rS_p^m)U^b(S_p^*)^nU^d U^jS_p^h\\
&=\sum_{j=0}^{p^h-1}\sum_{i_2=0}^{p^{h-m}-1}e_{r+i_2p^m,j}\otimes (S_p^*)^{h-m}   U^{-i_2+b} (S_p^*)^nU^d U^jS_p^h \\
&=\sum_{j=0}^{p^h-1} \sum_{i_2=0}^{p^{h-m}-1}e_{r+i_2p^m,j}\otimes (S_p^*)^{h-m} U^{-i_2+b} ((S_p^*)^nU^d U^jS_p^n) S_p^{h-n}\\
&=\sum_{j_1=0}^{p^n-1}\sum_{j_2=0}^{p^{h-n}-1} \sum_{i_2=0}^{p^{h-m}-1}e_{r+i_2p^m,j_1+p^{n}j_2}\otimes (S_p^*)^{h-m} U^{-i_2+b} ((S_p^*)^nU^d U^{j_1+p^{n}j_2}S_p^n) S_p^{h-n}\\
&=\sum_{j_1=0}^{p^n-1}\sum_{j_2=0}^{p^{h-n}-1} \sum_{i_2=0}^{p^{h-m}-1}e_{r+i_2p^m,j_1+p^{n}j_2}\otimes (S_p^*)^{h-m} U^{-i_2+b} ((S_p^*)^nU^d U^{j_1}S_p^n)U^{j_2} S_p^{h-n}\\
\end{align*}
\begin{align*}
&=\sum_{j_2=0}^{p^{h-n}-1} \sum_{i_2=0}^{p^{h-m}-1}e_{r+i_2p^m,p^n-d+p^{n}j_2}\otimes (S_p^*)^{h-m} U^{-i_2+b+1} U^{j_2} S_p^{h-n}\\
&=\sum_{j_2=0}^{p^{h-n}-1} \sum_{i_2=0}^{p^{h-m}-1}e_{r+i_2p^m,p^n-d+p^{n}j_2}\otimes (S_p^*)^{h-m} U^{-i_2+b+j_2+1} S_p^{h-m}S_p^{m-n}\\
&=\sum_{j_3=0}^{p^{h-m}-1}\sum_{j_4=0}^{p^{m-n}-1} \sum_{i_2=0}^{p^{h-m}-1}e_{r+i_2p^m,p^n-d+p^{n}(j_3+j_4p^{h-m})}\otimes (S_p^*)^{h-m} U^{-i_2+b+1} U^{j_3+j_4p^{h-m}} S_p^{h-m}S_p^{m-n}\\
&= \sum_{j_3=0}^{p^{h-m}-1}\sum_{j_4=0}^{p^{m-n}-1} \sum_{i_2=0}^{p^{h-m}-1}e_{r+i_2p^m,p^n-d+p^{n}(j_3+j_4p^{h-m})}\otimes (S_p^*)^{h-m} U^{-i_2+b+1} U^{j_3} S_p^{h-m}U^{j_4}S_p^{m-n}
\end{align*}
Now since $-2p^{h-m}+2\leq-i_2+b+1+j_3  \leq p^{h-m}$, there are two nontrivial multiples of
$p^{h-m}$ among the values taken by  $-i_2+b+j_3$, namely $\pm p^{h-m}$ itself. Therefore, the last expression we had can be rewritten as
\begin{align*}
&\sum_{j_4=0}^{p^{m-n}-1} \sum_{i_2=0}^{p^{h-m}-1}e_{r+i_2p^m,p^n-d+p^{n}(i_2-b-1+j_4p^{h-m})}\otimes U^{j_4}S_p^{m-n} \\
&\qquad + \sum_{j_4=0}^{p^{m-n}-1} \sum_{i_2=0}^{p^{h-m}-1}e_{r+i_2p^m,p^n-d+p^{n}(i_2-b-1-p^{h-m}+j_4p^{h-m})}\otimes U^{j_4-1}S_p^{m-n} \\
&\qquad + \sum_{j_4=0}^{p^{m-n}-1} \sum_{i_2=0}^{p^{h-m}-1}e_{r+i_2p^m,p^n-d+p^{n}(i_2-b-1+p^{h-m}+j_4p^{h-m})}\otimes U^{j_4+1}S_p^{m-n} 
\end{align*}

We now move on to treat the second case, namely when $m=n$. There is no loss of generality if we also suppose  $|a|, |d|< p^h$ and $|d|<p^m$. As in the first case, there are again four subcases to consider depending on the signs
of $a$ and $d$. As they are very similar to one another, we treat in full one of these only:  the case when
 $a\geq 0$ and $d\leq 0$. 
We observe that $a=p^mb+r$ (with $0\leq r<p^m$).
Note also that $0\leq b<p^{h-m}$. 
Then we have
\begin{align*}
\Psi_h(x)&=\sum_{i,j=0}^{p^h-1} e_{i,j}\otimes (S_p^*)^{h}U^{-i} U^aS_p^m(S_p^*)^mU^d U^jS_p^h\\
&=\sum_{i,j=0}^{p^h-1} e_{i,j}\otimes (S_p^*)^{h}U^{-i} U^rS_p^mU^b(S_p^*)^mU^d U^jS_p^h\\
&=\sum_{i,j=0}^{p^h-1} e_{i,j}\otimes (S_p^*)^{h-m} ((S_p^*)^{m}U^{-i} U^rS_p^m)U^b(S_p^*)^mU^d U^jS_p^h\\
\end{align*}
\begin{align*}
&=\sum_{j=0}^{p^h-1}\sum_{i_1=0}^{p^m-1}\sum_{i_2=0}^{p^{h-m}-1} e_{i_1+i_2p^m,j}\otimes (S_p^*)^{h-m} ((S_p^*)^{m}U^{-(i_1+i_2p^m)} U^rS_p^m)U^b(S_p^*)^mU^d U^jS_p^h\\
&=\sum_{j=0}^{p^h-1}\sum_{i_1=0}^{p^m-1}\sum_{i_2=0}^{p^{h-m}-1} e_{i_1+i_2p^m,j}\otimes (S_p^*)^{h-m} U^{-i_2} ((S_p^*)^{m}U^{-i_1} U^rS_p^m)U^b(S_p^*)^mU^d U^jS_p^h\\
&=\sum_{j=0}^{p^h-1}\sum_{i_2=0}^{p^{h-m}-1}e_{r+i_2p^m,j}\otimes (S_p^*)^{h-m}   U^{-i_2+b} (S_p^*)^mU^d U^jS_p^h \\
&=\sum_{j=0}^{p^h-1} \sum_{i_2=0}^{p^{h-m}-1}e_{r+i_2p^m,j}\otimes (S_p^*)^{h-m} U^{-i_2+b} ((S_p^*)^mU^{d} U^jS_p^m) S_p^{h-m}\\
&=\sum_{j_1=0}^{p^m-1}\sum_{j_2=0}^{p^{h-m}-1} \sum_{i_2=0}^{p^{h-m}-1}e_{r+i_2p^m,j_1+p^{m}j_2}\otimes (S_p^*)^{h-m} U^{-i_2+b} ((S_p^*)^mU^d U^{j_1+p^{m}j_2}S_p^m) S_p^{h-m}\\
&=\sum_{j_2=0}^{p^{h-m}-1} \sum_{i_2=0}^{p^{h-m}-1}e_{r+i_2p^m,-d+p^{m}j_2}\otimes (S_p^*)^{h-m} U^{-i_2+b+j_2} S_p^{h-m}\\
\end{align*}
Since $-p^{h-m}+1\leq-i_2+b+j_2\leq 2p^{h-m}-2$, $p^{h-m}$ is the only nontrivial multiple
among the possible values of  $-i_2+b+j_2$, which means the above expression rewrites as
\begin{align*}
&\sum_{j_2=0}^{p^{h-m}-1}e_{r+(j_2+b)p^m,-d+p^{m}j_2}\otimes 1+ \sum_{j_2=0}^{p^{h-m}-1}e_{r+(j_2+b+p^{h-m})p^m,-d+p^{m}j_2}\otimes U
\end{align*}

Finally, the case $m<n$ requires no work, for it is easily reconducted  to the first   thanks to the the fact that $\Psi_h$ is a $*$-homomorphism.

The  inequalities involving the norms follow from the formulas arrived at above. Indeed, set $A:=R_{0}\otimes S_p^{m-n}+\tilde R_{0}\otimes S_p^{m-n}U+\hat R_{0}\otimes U^*S_p^{m-n}$. When $m>n$, we have
\begin{align*}
\|\Psi_h(x)\|^2&=\| \left( \sum_{j_4=1}^{p^{m-n}-1} R_{j_4}\otimes U^{j_4}S_p^{m-n}+A\right)^* \left(\sum_{j_4=1}^{p^{m-n}-1} R_{j_4}\otimes U^{j_4}S_p^{m-n}+A\right)\|\\
&=\| \sum_{j_4=1}^{p^{m-n}-1} R_{j_4}^*R_{j_4}\otimes 1+A^*A \| 
\end{align*}
It follows that 
\begin{align*}
\|x\|^2=\|\Psi_h(x)\|^2&\geq \| R_{j_4}^*R_{j_4} \|=\|R_{j_4}\|^2 \qquad \forall j\in\{1, \ldots , p^{m-n}-1\}\\
\|x\|^2&\geq \| A^*A\| =\|A\|^2
\end{align*}
%
From Lemma \ref{normineq} we get $\|x\|\geq \max\{\| R_0\|, \| \tilde R_0\|, \| \hat R_0\|\}$.
The cases $m<n$ and $m=n$ are quite analogous.
\end{proof}

Going back to the computations of the entropy of our winding endomorphisms, the first thing
we do is to provide a lower bound for it.  As in \cite{SZ}, this can be done by looking at the restriction of $\chi_k$ to a suitable MASA of $\CQ_p$. In
our case, the convenient  MASA  to  consider is obviously  $C^*(U)$ (cf. \cite{ACR, ACRS}).
\begin{lemma}\label{lowerbound}
For any integer $k$ coprime with $p$, one has
$$
{\rm ht}(\chi_{k}) \geq \log|k|
$$
\end{lemma}
\begin{proof}
The claim follows by monotonicity 
$$
{\rm ht}(\chi_{k})\geq {\rm ht}(\chi_{k})\upharpoonright_{C^*(U)}= h^{top}(T_k)= \log|k|\; .
$$
where $T_k(z):=z^k$,
see e.g. \cite{Block}.
\end{proof}
We are now in a position to prove the main result of this paper.

\begin{proof}[Proof of Theorem \ref{entropy}.]
Thanks to Lemma \ref{lowerbound}, all we have to do is   show that  the entropy of $\chi_k$ is less than or equal to
$\log|k|$.
For any $l\in\IN$, we set $\omega_l:=\cup_{q,r,s=0}^l\CA_{q,r,s}$. For $n\in \IN$ we denote by $\omega^{(n)}_l$ the  union $\cup_{j=0}^n\chi_{k}^j(\omega_l)$.
Fix $\delta>0$. Since $\CQ_p$ is nuclear, there exists $(\phi_0,\psi_0, M_{C_l}(\IC))\in {\rm CPA}(\CQ_p,\omega_l,\frac{\delta}{16\cdot p^l})$.
 We want to find an $m$ such that the exponents of $U$
appearing in the elements in $\Psi_m(\omega_l^{(n)})\subset M_{p^m}(\CQ_p)$ are smaller than $p^m$. 
This is certainly the case  provided that $|k|^nl<p^m$, as follows from a straightforward 
application of Remark \ref{exponents}. The inequality can also be rewritten as $n\log_p|k|+\log_p(l)<m$,
which is more suited to our purposes.
For instance, we can simply choose $m=[n\log_p|k|+\log_p(l)]+1$, where $[\cdot]$ denotes the integer part of a real number.\\
Again, by nuclearity of $\CQ_p$  there exists a $d\in\IN$ and u.c.p. map $\gamma: \Psi_m(\CQ_p)=M_{p^m}(\CQ_p)\to M_d(\IC)$ and $\eta:   M_d(\IC)\to \CQ_p$ such that for all $a\in\omega_l^{(n)}$ the inequality
$
\| \eta\circ \gamma(\Psi_m(a))-a\|<\delta/2
$ holds.\\
Set $\psi:=({\rm id}\otimes\psi_0)\circ\Psi_m$ and $\phi:=\eta\circ\gamma\circ({\rm id}\otimes \phi_0)$.
Now for any $x\in \omega_l$ and $h\in\IN$ with $h\leq n$, we have
\begin{align}\label{inequality-proof}
\begin{array}{l}
\| \phi\circ\psi(\chi_{k}^h(x))-\chi_{k}^h(x)\|=\\
=\| \eta\circ\gamma\circ({\rm id}\otimes \phi_0\circ\psi_0)\circ\Psi_m(\chi_{k}^h(x))-(\chi_{k}^h(x)\|\\
=\| \eta\circ\gamma\circ({\rm id}\otimes \phi_0\circ\psi_0)\circ\Psi_m(\chi_{k}^h(x))-\eta\circ\gamma\circ\Psi_m(\chi_{k}^h(x))\| \\
\qquad + \| \eta\circ\gamma\circ\Psi_m(\chi_{k}^h(x))-(\chi_{k}^h(x)\|\\
\leq \| \eta\circ\gamma\circ({\rm id}\otimes \phi_0\circ\psi_0)\circ\Psi_m(\chi_{k}^h(x))-\eta\circ\gamma\circ\Psi_m(\chi_{k}^h(x))\| +\frac{\delta}{2}\\
\leq \|({\rm id}\otimes (\phi_0\circ\psi_0))\circ\Psi_m(\chi_{k}^h(x))-\Psi_m(\chi_{k}^h(x)\| +\frac{\delta}{2}
\end{array}
\end{align}
By Remark \ref{exponents} $\chi_{k}^h(x)$ is in $\CB_{q,r,s}$ where $q\leq |k|^{h}l$ (which is smaller than $p^m$ by construction), and $r, s\leq l$.
As of now we will also assume $r>s$ since the case $r\leq s$ can be handled by means of similar computations.
By Inequality  \eqref{inequality-proof} and  Lemma \ref{lemma21}  we get
\begin{align*}
&\| \phi\circ\psi(\chi_{k}^h(x))-\chi_{k}^h(x)\|\\
&< \|\sum_{j_4=1}^{p^{r-s}-1} R_{j_4}\otimes (\phi_0\circ\psi_0)(U^{j_4}S_p^{r-s})+\\
&\qquad +R_{0}\otimes (\phi_0\circ\psi_0)(S_p^{r-s})+\tilde R_{0}\otimes (\phi_0\circ\psi_0)(S_p^{r-s}U)+\hat R_{0}\otimes (\phi_0\circ\psi_0)(U^*S_p^{r-s})\\
&\qquad -\sum_{j_4=1}^{p^{r-s}-1} R_{j_4}\otimes U^{j_4}S_p^{r-s}-R_{0}\otimes S_p^{r-s}-\tilde R_{0}\otimes S_p^{r-s}U-\hat R_{0}\otimes U^*S_p^{r-s}\| +\frac{\delta}{2}\\
&= \|\sum_{j_4=1}^{p^{r-s}-1} (R_{j_4}\otimes (\phi_0\circ\psi_0)(U^{j_4}S_p^{r-s})-U^{j_4}S_p^{r-s})+R_{0}\otimes ((\phi_0\circ\psi_0)(S_p^{r-s})-S_p^{r-s})\\
&\qquad +\tilde R_{0}\otimes ((\phi_0\circ\psi_0)(S_p^{r-s}U)-S_p^{r-s}U)+\\
&\qquad +\hat R_{0}\otimes ((\phi_0\circ\psi_0)(U^*S_p^{r-s})-U^*S_p^{r-s})\| +\frac{\delta}{2}\\
&\leq\sum_{j_4=1}^{p^{r-s}-1} \|(R_{j_4}\otimes (\phi_0\circ\psi_0)(U^{j_4}S_p^{r-s})-U^{j_4}S_p^{r-s})\|\\
&\qquad +\|R_{0}\otimes ((\phi_0\circ\psi_0)(S_p^{r-s})-S_p^{r-s})+\tilde R_{0}\otimes ((\phi_0\circ\psi_0)(S_p^{r-s}U)-S_p^{r-s}U)+\\
&\qquad +\hat R_{0}\otimes ((\phi_0\circ\psi_0)(S_p^{r-s}U)-S_p^{r-s}U)\| +\frac{\delta}{2}\\
&\leq \frac{\delta}{16\cdot p^l} \left(\sum_{j_4=1}^{p^{r-s}-1}  1\right) +\frac{\delta}{16\cdot p^l}3+\frac{\delta}{2}\\
&\leq \frac{\delta}{16\cdot p^l} (p^{r-s}-1)  +\frac{3\delta}{16\cdot p^l}+\frac{\delta}{2}\\
&\leq \frac{\delta}{16}   +\frac{3\delta}{32}+\frac{\delta}{2} =\frac{19\delta}{32}
<\delta
\end{align*}
where we used that $\|R_{j_i}\|\leq 1$ for all $i$, $\| R_0\|\leq 1$, $\| \tilde R_0\|\leq 1$, $\| \hat R_0\|\leq 1$.
The inequality proved above  shows that the triple $(\phi,\psi, M_{p^m}(\IC)\otimes M_{C_l}(\IC))$ is in CPA$(\CQ_p,\omega_l^{(n)}, \delta)$. Therefore,   rcp$(\omega_l^{(n)},\delta)$ must be less than or equal to $C_lp^m$. Accordingly, the natural logarithm of the  former quantity can be bounded in the following way:
\begin{align*}
\log {\rm rcp}(\omega_l^{(n)},\delta) & \leq \log(C_l)+m\log(p)=\log(C_l)+\log(p)(\left[n\log_p|k|+\log_p(l)\right]+1)\\
&= \log(C_l)+\log(p)\left(\frac{\log |k|}{\log(p)} n+\log_p(l)\right)+2\log(p)\\
&\leq \log(C_l)+n\log|k|+\log(l)+2\log(p),
\end{align*}
from which we find
$
\limsup_{n\to\infty}\left( \frac{1}{n} \log {\rm rcp}(\omega_l^{(n)},\delta)\right) \leq  \log |k|
$.\\
The thesis is thus arrived at thanks to the Kolmogorov-Sinai property of non-commutative entropy.
\end{proof}
It is worth noting that all winding endomorphisms leave the Bunce-Deddens algebra
$\CQ_p^\IT$ invariant, which means one can also compute the entropy of
the restriction of the winding endomorphisms to this subalgebra. It turns out that the index
of the restriction does not decrease. Indeed, we have the following
\begin{corollary}\label{entropyBD}
For any integer $k$ coprime with $p$, one has
$$
{\rm ht}(\chi_{k}\upharpoonright_{\CQ_p^\IT})= \log|k|
$$
\end{corollary}
\begin{proof}
The claim follows directly from
 the monotonicity of the entropy:
$$
\log|k|={\rm ht}(\chi_{k}\upharpoonright_{C^*(U)})\leq{\rm ht}(\chi_{k}\upharpoonright_{\CQ_p^\IT})\leq{\rm ht}(\chi_{k})= \log|k|.
$$
\end{proof}

\section{On the Watatani index of the winding endomorphisms}\label{secindex}


Motivated by the work done in \cite{CS} on quadratic permutation endomorphisms of the Cuntz algebra $\CO_2$, in this section we undertake a study of the relation between the entropy and index of the restriction to the Bunce-Deddens subalgebras of our winding endomorphisms.\\
We are going to show that, as well as the entropy, the Watatani index (see \cite{W} for the definition and the main properties) of the restriction of $\chi_k$ to the Bunce-Deddens
algebra $\CQ_p^\IT$ can also be computed exactly. More precisely, for any integer $k$ coprime with 
$p$, the value of the index turns out to be $|k|$. Rather interestingly, the entropy of $\chi_k\upharpoonright_{\CQ_p^\IT}$ is then given by the
natural logarithm of the index of $\chi_k\upharpoonright_{\CQ_p^\IT}$.\\

For the reader's convenience we recall some basic definitions that we will make use of. 
We start with an inclusion of unital $C^*$-algebras $\CA\subset\CB$ with a common unit $I$ such that there exists a faithful conditional expectation
$E:\CB\rightarrow\CA$.
\begin{definition}
A finite family $\{u_1, u_2, \ldots, u_n\}\subset \CB$ is said to be a quasi-basis for $E$ if for any
$x\in\CB$ one has
$$x=\sum_{i=1}^n u_iE(u_i^*x)=\sum_{i=1}^n E(xu_i)u_i^*$$
\end{definition} 
Now our conditional expectation $E:\CB\rightarrow\CA$ has finite index if there exists a quasi-basis for
$E$. The index is then defined as ${\rm Ind}(E):=\sum_{i=1}^n u_iu_i^*\in\CB$.
Despite its definition, the index does not depend on the chosen quasi-basis. Moreover, 
the element ${\rm Ind}(E)$ sits in the center of $B$. In particular, if $B$ has trivial center, then 
${\rm Ind}(E)$ is a positive real number greater than or equal to $1$.
The Watatani index of the inclusion $\CA\subset\CB$ is then  defined as the infimum of the set of the indices obtained as above
corresponding to any conditional expectation of finite index. It turns out that this infimum is actually a minimum
provided that both $\CB$ and $\CA$ have trivial center, see \cite{W}.\\

To accomplish our computation of the index, we will make use of the following key result, where the image of the winding
endomorphism (restricted to $\CQ_p^\IT$) is seen to coincide with the fixed-point subalgebra of a finite-order
automorphism, which we next define.
Fix $z:=e^{2\pi i/k}$ and let $\alpha: \CQ_p^\IT\to\CQ_p^\IT$ be the automorphism mapping $U$ to $zU$ and fixing $S_p^h(S_p^*)^h$ for all $h\in\IN$, see e.g. \cite{Cuntzaxb}. Note that $\alpha^{|k|}={\rm id}_{\CQ_p^\IT}$, which allows us to define a conditional expectation $E$ from $\CQ_p^\IT$ to $(\CQ_p^\IT)^\alpha:=\{x\in\CQ_p^\IT: \alpha(x)=x\}$ as
\begin{equation}\label{ce}
E(x):=\frac{1}{|k|}\sum_{l=0}^{|k|-1}\alpha^n(x),\quad x\in\CQ_p^\IT.
\end{equation}

\begin{proposition}\label{index}
The image $\chi_k(\CQ_p^\IT)$ coincides with  $(\CQ_p^\IT)^\alpha$.
\end{proposition}
First of all, we need a preliminary lemma.
\begin{lemma}\label{lemma-Bar}
Let $p, k\in\IZ$ be coprime integers. For any $i, h\in \IN$,  there exist $b, m\in \IZ$ such that $i+bp^h=mk$.
\end{lemma}
\begin{proof}
Since $k$ and $p$ are coprime, so are $k$ and $p^h$.
This means that $1=ck+dp^h$ for some $c, d\in\IZ$.
It follows that $i=ick+idp^h$ and so we may choose $b$ as $-id$ and $m=ic$.
\end{proof}
We now go back to the proof of Proposition \ref{index}.
\begin{proof}[Proof of Proposition \ref{index}]
Since $\chi_k(\CQ_p^\IT)$ is generated by the monomials of the form $U^{ki}S_p^h (S_p^*)^hU^{kj}$, with $i, j\in\IZ$, $h\in\IN$, the inclusion $\chi_k(\CQ_p^\IT)\subset (\CQ_p^\IT)^\alpha$ is clear.\\
For the converse inclusion, we  need to use the  conditional expectation $E: \CQ_p^\IT \to (\CQ_p^\IT)^\alpha$ 
defined in \eqref{ce}.
As already mentioned, the algebra $\CQ_p^\IT$ is linearly generated by the elements of the form $U^{i}S_p^h (S_p^*)^hU^{j}$, $h\in\IN$ and $i, j\in\IZ$. So all we have to do is show that the images of these elements under $E$ are in $\chi_k(\CQ_p^\IT)$.
A straightforward computation shows that $E(U^{i}S_p^h (S_p^*)^hU^{j})$ is zero, unless $i+j$ is $0$ mod $k$, that is $j=-i +mk$ for some $m\in\IZ$. 
If $E(U^{i}S_p^h (S_p^*)^hU^{j})$ is not zero, then it is equal to $U^{i}S_p^h (S_p^*)^hU^{-i} U^{km}$.
Now $U^{i}S_p^h (S_p^*)^hU^{-i} U^{km}$ lies in $\chi_k(\CQ_p^\IT)$ if and only if $U^{i}S_p^h (S_p^*)^hU^{-i}$ does.
Thanks to the chain of equalities
$$
U^{i}S_p^h (S_p^*)^hU^{-i}=U^{i}S_p^hU^bU^{-b} (S_p^*)^hU^{-i}=U^{i+bp^h}S_p^h (S_p^*)^hU^{-i-bp^h}
$$
which hold for any $b\in\IN$, our claim finally follows from Lemma \ref{lemma-Bar}.
\end{proof}

\begin{remark}
The inclusion  $(\CQ_p^{\IT})^\alpha\subset \CQ_p^{\IT}$ is an example of a non-commutative self-covering of the type considered in \cite{AGI}. 
\end{remark}

We are now in a position to compute the index of $\chi_k$ relative to the conditional expectation $E$ considered above.

\begin{theorem}\label{indexE}
The index of $\chi_k(\CQ_p^\IT)$ in $\CQ_p^\IT$ relative to the conditional expectation $E$ above is $|k|$.
\end{theorem}
\begin{proof}
It is a direct consequence of  the description of $\chi_k(\CQ_p^\IT)$ as the fixed-point algebra under the action of the finite group $\IZ_{|k|}$,
which has order $|k|$.
\end{proof}

Our next goal is to show that the conditional expectation $E$ we have used so far is actually unique. 
This will be a consequence of Corollary 1.4.3 in \cite{W} once we have 
ascertained the set equality
$\chi_k(\CQ_p^\IT)'\cap\CQ_p^\IT=\IC$.
\begin{proposition}
For any $k$ coprime with $p$, the relative commutant $\chi_k(\CQ_p^\IT)'\cap\CQ_p^\IT$ is trivial.
\end{proposition}
\begin{proof}
Since $\CD_p$ is contained in $\chi_k({\CQ_p^\IT})$, we have
 $\chi_k(\CQ_p^\IT)'\cap\CQ_p^\IT\subset \CD_p\cap C^*(U^k)'$. 
We are thus led to prove that any $x\in\CD_p$ such that $xU^k=U^kx$ is actually a scalar.
We will work in the canonical representation of $\CQ_p$, which acts on the Hilbert
space $\ell^2(\IZ)$.
Now any $x\in \CD_p$ is a diagonal operator w.r.t. the canonical basis
$\{e_i: i\in\IZ\}$ of $\ell^2(\IZ)$, that is $xe_i=x_ie_i$, $i\in\IZ$, for suitable
$x_i\in\IC$.
Since $U^ke_i= e_{i+k}$, for any $i\in\IZ$, it is easy to see that any such $x$ commutes with $U$ if and only if
$x_{i}=x_{i+hk}$ for any $i, h\in\IZ$. In particular, the set $\{x_i: i\in\IZ\}$ is finite.
In other terms, the spectrum of $x$ is finite as well, which means its spectral projections
belong to $\CD_p$. Now write $x=\sum_{i=0}^{k-1}x_iP_i$, where $P_i$ is the orthogonal projection onto the subspace
$\overline{{\rm span}}\{ e_{i+hk}: h, k\in\IZ\}$. As there can exist different values of $i$ (in $\{0, 1, 2, ..., k-1\}$) giving
the same $x_i$, we rewrite the above sum as
$x:= \sum_{\lambda\in\sigma(x)} Q_\lambda$, where $Q_\lambda$ is the spectral projection associated with $\lambda$ and
$Q_\lambda:= \sum_{i: x_i=\lambda} P_i$.
But because $\sigma(x)$ is finite, each $Q_\lambda$ can be obtained via the continuous functional calculus of $x$, which
means $Q_\lambda$ belongs to $\CD_p$ for every $\lambda\in\sigma(x)$.
As we next show, this implies that $x$ must be a multiple of the identity. For, if it is not a scalar, then
$\sigma(x)$ contains at least two different values, say $\lambda$ and $\mu$. Now any projection in $\CD_p$ is a finite sum of projections of the type $U^m(S_p)^n(S_p^*)^nU^{-m}$, with
$m, n\in\IN$, {\it cf.} \cite[Lemma 6.21]{ACR}. In particular,
$Q_\lambda\geq U^m(S_p)^n(S_p^*)^nU^{-m}$ and
 $Q_\mu\geq U^{m'}(S_p)^n(S_p^*)^nU^{-m'} $ for some $n, m, m'\in\IN$ (there is no loss of generality is assuming that the power $n$ of the isometry $S_p$ is the same in the two inequalities). 
Now take $a:=p^nh_1+m$ and $b:=p^nh_2+ m'$.
Observe that by construction if $e_n$ lies in the range of $Q_\lambda$, so does $e_{n+k}$. We claim that
$a+lk\equiv b$ mod $p^n$ for some $l\in\IZ$. From this we find $Q_\lambda e_{a+lk}=Q_\mu e_{a+lk}=e_{a+lk}$, which is absurd since
$Q_\lambda Q_\mu=0$. The claim follows from the fact that $k$ and $p^n$ are coprime: multiplying
$ks +p^nt=1$ by $(b-a)$ we get $(b-a)sk+t(b-a)p^n=b-a$, that is
$a + (b-a)sk=b-t(b-a)p^n$ and the claim is thus verified with $l=(b-a)s$. The proof is complete.
\end{proof}
\begin{remark}
From the proof of the foregoing result we can single out the interesting information that all
powers $T^k$ of the $p$-adic odometer $T$ on the Cantor set $K$ with $k$ and $p$ coprime enjoy the following property:
any continuous $T^k$-invariant function is a constant.
\end{remark}
As an application of the previous result we also find the following.
\begin{theorem}\label{indexBunceDeddens}
For any $k$ coprime with $p$, the Watatani index of $\chi_k(\CQ_p^\IT)$ in $\CQ_p^\IT$ is $|k|$.
\end{theorem}
\begin{proof}
Since the conditional expectation $E:\CQ_p^\IT\rightarrow \chi_k(\CQ_p^\IT)$ is unique, the index we computed in
Theorem \ref{indexE} is actually the Watatani index of the inclusion $\chi_k(\CQ_p^\IT)\subset\CQ_p^\IT$.
\end{proof}

We would be inclined to believe that the index of $\chi_k$ is still $|k|$ at the level of the whole
$p$-adic ring $C^*$-algebras. Despite our efforts, though, we have not been able to ascertain this equality in full generality.
Even so, we do know the index is $|k|$ when $k=\pm(p-1)^i$, $i\geq 1$ (note that $p$ and $p-1$ are always coprime).

We start our analysis with $k=p-1$. In this case by universality
it is not difficult to see that for any $z\in\IT$ such that $z^{p-1}=1$,
$\beta(S_p):=S_p$ and $\beta(U):=zU$ defines an automorphism of $\CQ_p$.
Again, $\beta$ has finite order in that $\beta^{p-1}={\rm id}_{\CQ_p}$.
In the following we take $z$ as a primitive root of unity of order $p-1$ so that the corresponding $\beta$ gives an automorphic action
of $\IZ_{p-1}$.
Let now $F$ be the conditional expectation from $\CQ_p$ to $\CQ_p^\beta$
given by $F:=\frac{1}{p-1}\sum_{i=0}^{p-2}\beta^i$. We have the following result.
\begin{theorem}
Let $p\geq 2$ be a fixed integer. If $k=p-1$, then 
$\chi_k(\CQ_p)=\CQ_p^{\beta}$. In particular, the index of
$\chi_{(p- 1)}(\CQ_p)$ relative to $F$ is $p- 1$.
\end{theorem}
\begin{proof}
The equality $\chi_k(\CQ_p)=\CQ_p^\beta$ is trivially satisfied. By the same argument as in the proof of Theorem \ref{indexE} it follows that the index of $\chi_k$ is equal to $p-1=k$.
\end{proof}
\begin{remark}
The case of a negative $k$, say $k=-(p-1)$ easily follows from the above result as $\chi_{-k}=\chi_{-1}\circ\chi_k$ and
the index of $\chi_{-1}$ is $1$ because $\chi_{-1}$ is an automorphism.
\end{remark}

In order to conclude that the value of the
 index determined above is actually the Watatani index, we need to prove that
$F$ is the only conditional expectation of $\CQ_p$ onto $\CQ_p^\beta$.

\begin{proposition}
The conditional expectation $F: \CQ_p\rightarrow \CQ_p^\beta$ is unique.
\end{proposition}
\begin{proof}
Again, we need only prove that $(\CQ_p^\beta)'\cap \CQ_p= \IC$, which is a consequence of the equality
$C^*(S_p)'\cap\CQ_p=\IC$ (see \cite[Theorem, 4.6]{ACRS}) since $\beta(S_p)=S_p$. 
\end{proof}

As a result, the following is now straightforward.

\begin{theorem}
For any $p\geq 2$, the Watatani index  of $\chi_k(\CQ_p)$ in $\CQ_p$ is $|k|$ if
$k=p-1$.
\end{theorem}

Now iterating the procedure above, it is clear that there exists a conditional
expectation $F_j$ form $\chi_{p-1}^j(\CQ_p)$ onto  $\chi_{p-1}^{j+1}(\CQ_p)$ and the index of $F_j$ is still
$p-1$. Compounding these conditional expectations, for every $i\geq 1$ one obtains a conditional expectation from
$\CQ_p$ onto $\chi_{(p-1)}^i=\chi_{(p-1)^i}$ whose index is obviously $(p-1)^i$.
Again, the relative commutant $\chi_{(p-1)^i}(\CQ_p)'\cap\CQ_p$ is trivial as $\chi_{(p-1)^i}(S_p)=S_p$,  and so the conditional expectation is unique. Therefore, collecting everything together, the following result is arrived at.

\begin{theorem}
For any $k=\pm(p-1)^i$, $i\in\IN$, the Watatani index of $\chi_k(\CQ_p)$ in $\CQ_p$ is equal to $k$.
\end{theorem}

\section*{Acknowledgements}
V. A. acknowledges the support by the Swiss National Science foundation through the SNF project no. 178756 (Fibred links, L-space covers and algorithmic knot theory).
The authors wish to thank Simone Del Vecchio for several useful discussions on the topic of this paper.


\begin{thebibliography}{99}

\bibitem{AKM} R. L. Adler, A.G. Konheim, M. H. McAndrew, \emph{Topological entropy}, Trans. Amer. Math. Soc.
\textbf{114} (1965), 309--319.
\bibitem{ACR} V. Aiello, R. Conti, S. Rossi, \emph{A look at the inner structure of the $2$-adic ring $C^*$-algebra and its automorphism groups}, Publ. Res. Inst. Math. Sci. \textbf{54} (2018), 45--87.
\bibitem{ACR2} V. Aiello, R. Conti, S. Rossi, \emph{Diagonal automorphisms of the $2$-adic ring $C^*$-algebra}, Q. J. Math. \textbf{69} (2018), 815--833.  
\bibitem{ACR3} V. Aiello, R. Conti, S. Rossi, \emph{Permutative representations of the $2$-adic ring $C^*$-algebra}, J. Operator Theory  \textbf{82} (2019), 197--236.
\bibitem{ACR4} V. Aiello, R. Conti, S. Rossi, \emph{Normalizers and permutative endomorphisms of the $2$-adic ring $C^*$-algebra}. J. Math. Anal. Appl. \textbf{481} (2020),  123395, 25 pp.
\bibitem{ACR5} V. Aiello, R. Conti, S. Rossi, \emph{A Fej\'er theorem for boundary quotients arising from algebraic dynamical systems}, 
Annali della Scuola Normale Superiore di Pisa, Classe di Scienze, \textbf{22} (2021),  305-313, 
DOI  10.2422/2036-2145.201903\_007, preprint arXiv:1911.03414
\bibitem{ACR6} V. Aiello, R. Conti, S. Rossi, \emph{A Hitchhiker's Guide to Endomorphisms and Automorphisms of Cuntz Algebras}, Rend. Mat. Appl. \textbf{42} (2021), 61--162.

\bibitem{ACRS}  V. Aiello, R. Conti, S. Rossi,  N. Stammeier, \emph{The inner structure of boundary quotients of right LCM semigroups}, Indiana Univ. Math. J. \textbf{69} (2020), 1627--1661. 


\bibitem{AGI} V. Aiello, D. Guido, T. Isola, \emph{Spectral triples for noncommutative solenoidal spaces from self-coverings},  J. Math. Anal. Appl.  \textbf{448} (2017), 1378--1412.  

\bibitem{AR1} V. Aiello, S. Rossi, \emph{On the cyclic automorphism of the Cuntz algebra and its fixed-point algebra}. 
J. Math. Anal. Appl. \textbf{505}, Issue 1,  125476 (2022).

\bibitem{BOSS} S. Barlak, T. Omland,  N. Stammeier, \emph{On the $K$-theory of C$^*$-algebras arising from integral dynamics,} Ergod. Theory Dyn. Syst. \textbf{8} (2018), 832--862.

\bibitem{Block} L. Block, J. Guckenheimer, M. Misiurewicz and L.-S. Young, \emph{Periodic points and topological entropy of one dimensional maps}, in Global Theory of Dynamical Systems, Lecture Notes in Math., 819, Springer, Berlin, 1980, 18--34.
\bibitem{Boca} F. Boca, P. Goldstein, \emph{Topological entropy for the canonical endomorphisms of Cuntz-Krieger algebras}, Bull. London Math. Soc. \textbf{32} (2000), 345--352.
\bibitem{B} N. Brown, \emph{Topological entropy in exact $C^*$-algebras}, Math. Ann. \textbf{314} (1999), 347--367.
\bibitem{Choda} M. Choda, \emph{Entropy of Cuntz's canonical endomorphism}, Pacific J. Math. \textbf{190} (1999), 235--245.
\bibitem{CS} R. Conti, W. Szyma\'nsky, \emph{Computing the Jones index of quadratic permutation endomorphisms of $\CO_2$}, J. Math. Phys. \textbf{50} (2009), 012705.


\bibitem{Cuntz1} J. Cuntz, \emph{Simple C$^*$-algebras generated by isometries},  Commun. Math. Phys. \textbf{57} (1977), 173--185.
 \bibitem{Cuntzaxb} J. Cuntz, \emph{C$^*$-algebras associated with the $ax+b$ semigroup over $\IN$.} K-theory and noncommutative geometry 2 (2008): 201-215.

\bibitem{LarsenLi} N. S. Larsen and X. Li, \emph{The $2$-adic ring C*-algebra of the integers and its representations},  J. Funct. Anal. \textbf{262} (4) (2012), 1392--1426. 
\bibitem{HuaxinLin} H. Lin, \emph{An introduction to the classification of amenable $C^*$-algebras}, World Scientific Publishing Co., Inc., River Edge, NJ, (2001).

\bibitem{Adam} A. Skalski, \emph{On automorphisms of C$^*$-algebras whose Voiculescu entropy is genuinely non-commutative},  Ergodic Theory Dynam. Systems \textbf{31} (2011), 951--954.

\bibitem{SZ} A. Skalski, J. Zacharias, \emph{Noncommutative topological entropy of endomorphisms of Cuntz algebras}, Lett. Math. Phys. \textbf{86} (2008), 115--134.
\bibitem{SZ2} A. Skalski, J. Zacharias, \emph{Entropy of shifts on higher-rank graph C$^*$-algebras}. Houston J. Math., \textbf{34} (2008), 269--282.
\bibitem{Voi} D. Voiculescu, \emph{Dynamical approximation entropies and topological entropy in operator algebras}, Comm. Math. Phys. \textbf{170} (1995), 249--281.
\bibitem{W} Y. Watatani, \emph{Index for C$^*$-subalgebras}, Vol. 424 Am. Math. Soc., 1990.

\end{thebibliography}
\end{document}